\documentclass[reqno,onecolumn,letterpaper,12pt]{amsart}
\usepackage{amsmath,amssymb,fullpage,calc,enumerate}
\usepackage[ruled]{algorithm2e}
\usepackage{graphicx,psfrag,multicol}
\usepackage{rotating}
\usepackage{enumitem}
\usepackage{comment}
\usepackage[usenames]{color}
\usepackage{multicol}
\usepackage{tabularx}
\usepackage{booktabs}
\usepackage{comment}
\usepackage{tikz}
\usepackage{hyperref}

\DeclareGraphicsExtensions{.jpg,.eps,.mps,.png}
\newtheorem{theorem}{Theorem}
\newtheorem{lemma}[theorem]{Lemma}
\newtheorem{prop}[theorem]{Proposition}
\newtheorem{corollary}[theorem]{Corollary}

\newtheorem*{rem}{Remark}
\theoremstyle{definition}
\newtheorem{definition}{Definition}

\def\NN{\mathbb{N}}
\def\r{\mathbb{R}}

\def\0{\bf 0}

\def\f{\mathcal{F}}

\newcommand{\ra}{\rightarrow}
\newcommand{\diag}{\operatorname{diag}}
\def\base{\mathcal{B}}
\def\M{\mathcal{M}}
\def\inds{\mathcal{I}}
\newcommand{\udot}{\mathbin{\dot{\cup}}}
\newcommand{\rank}{\operatorname{rk}}
\newcommand{\kk}{k}
\newcommand{\clump}{\operatorname{Cp}}

\DeclareMathOperator*{\argmin}{\arg\!\min}
\newcommand{\transp}[1]{#1^{T}}
\newcommand{\copic}{\textsc{COPIC}}
\newcommand{\lcop}{\textsc{LCOP}}
\newcommand{\mplcop}{\textsc{MPLCOP}}
\newcommand{\umat}[2]{\mathcal{U}_{#1}^{#2}}
\newcommand{\pmatch}{\mathcal{PM}}
\newcommand{\paths}{\mathcal{P}}
\newcommand{\cuts}{\mathcal{CUT}}
\newcommand{\runtime}{\operatorname{T}}

\tikzstyle{vertex}=[circle,fill=black!25,minimum size=3pt,inner sep=0pt]
\tikzstyle{selected vertex} = [vertex, fill=red!24]
\tikzstyle{edge} = [draw,thick,-]
\tikzstyle{weight} = [font=\tiny]
\tikzstyle{selected edge} = [draw,line width=1pt,-,red!50]
\tikzstyle{ignored edge} = [draw,line width=1pt,-,black!20]
\newcommand{\setzo}{\{0,1\}}

\begin{document}

\title{Combinatorial Optimization problems with interaction costs: complexity and solvable cases}

\author{Stefan Lendl}
\author{Ante \'Custi\'c}
\author{Abraham P. Punnen}

\address{Department of Mathematics, Simon Fraser University Surrey, Central
City, 250-13450 102nd AV, Surrey, British Columbia,V3T 0A3, Canada}
\email{apunnen@sfu.ca, acustic@sfu.ca}

\address{Institute of Discrete Mathematics, Graz University of Technology, 
\newline Steyrergasse 30, 8010 Graz, Austria}
\email{lendl@math.tugraz.at}

\begin{abstract}
We introduce and study the combinatorial optimization problem with interaction costs (COPIC). 
COPIC is the problem of finding two combinatorial structures, one
from each of two given families, such that the sum of their independent linear
costs and the interaction costs between elements of the two selected structures
is minimized. 
COPIC generalizes the quadratic assignment problem and many other well studied
combinatorial optimization problems, and hence covers many real world
applications.
We show how various topics from different areas in the literature can be
formulated as special cases of COPIC. 
The main contributions of this paper are results on the computational complexity
and approximability of COPIC for different families of combinatorial structures
(e.g.\@ spanning trees, paths, matroids), and special structures of the
interaction costs. More specifically, we analyze the complexity if the
interaction cost matrix is parameterized by its rank and if it is a diagonal
matrix.
Also, we determine the structure of the intersection cost matrix, such that
COPIC is equivalent to independently solving linear optimization problems for
the two given families of combinatorial structures.
\end{abstract}


\thanks{This work was supported by NSERC discovery grant and an NSERC discovery accelerator supplement award awarded to Abraham P. Punnen}
\thanks{Stefan Lendl is supported by the Austrian Science Fund (FWF): W1230}

\maketitle

\section{Introduction}

Let a family $\f_1$ of subset of  $[m]=\{1,2,\ldots ,m\}$, and a family  $\f_2$ of subsets of $[n]=\{1,2,\ldots ,n\}$ represent feasible solutions.
We assume that $\f_1$ and $\f_2$ have a compact representation of size polynomial in
$m$ and $n$, respectively, although the number of feasible solutions in each family
could be of size exponential in $m$ or $n$. For each element $i\in [m]$ a linear cost
$c_i$ is given. Also, for each element $j\in [n]$ a linear cost $d_j$ is given. In
addition, for any $(i,j)\in [m]\times [n]$ their  \emph{interaction cost} $q_{ij}$ is given.  Then the
\emph{combinatorial optimization problem with interaction costs} (COPIC) is the
problem of finding $S_1\in \f_1$ and $S_2\in \f_2$ such that 
\begin{equation}
f(S_1,S_2) = \sum_{i\in S_1}\sum_{j\in S_2}q_{ij} + \sum_{i\in S_1}c_i + \sum_{j\in S_2}d_j
\end{equation}
is minimized. We denote an instance of this problem by
$\copic(\f_{1},\f_{2},Q,c,d)$, where $Q=(q_{ij})$ is the interaction cost matrix
and $c=(c_i)$, $d=(d_j)$ are linear cost vectors of the instance. This
generalizes the classical \emph {linear
cost combinatorial optimization problem}, where for a given family $\f$ of subsets of
$[n]$, and cost vector $w \in \r^{n}$ one tries to find a set $S \in \f$
minimizing
\[ \sum_{i \in S} w_{i}. \]
We denote an instance of this problem by $\lcop(\f, w)$.
\medskip

COPIC generalizes many well studied combinatorial optimization problems. For example, when $\f_1$ and $\f_2$ are respectively the family of perfect matchings in bipartite graphs $G_1$ and $G_2$ 
with respective edge sets $[m]$ and $[n]$, then COPIC reduces to the
\emph{bilinear assignment problem} (BAP)~\cite{CSPB16}. BAP is a generalization
of the well studied \emph{quadratic assignment problem}~\cite{C98} and the
\emph{three-dimensional assignment problem}~\cite{S00} and hence COPIC
generalizes these problems as well. When $\f_1$ and $\f_2$ contain all subsets
of $[m]$ and $[n]$ respectively, COPIC reduces to the
\emph{bipartite unconstrained quadratic programming problem}
\cite{Duarte2014123,  punnen2015bipartite,Glover2015,Karapetyan2017} studied in the literature by various
authors and under different names. Also, when $\f_1$ and $\f_2$ are feasible
solutions of  generalized upper bound constraints on $m$ and $n$ variables,
respectively, COPIC reduces to the \emph{bipartite quadratic assignment problem}
and its variations~\cite{CusticPunnenORL17,Punnen2016715}. Most quadratic combinatorial
optimization problems can also be viewed as special cases of COPIC, including
the \emph{quadratic minimum spanning tree problem}~\cite{AssadXu1992},
\emph{quadratic set covering problem}~\cite{BazaraaGoode1975}, \emph{quadratic
travelling salesman problem}~\cite{JagerMolitor2008}, etc. Thus all the
applications studied in the context of these special cases are applications of
COPIC as well. COPIC is a special case of bilinear integer programs
\cite{Konno81, Adams93, Freire12} when $\f_1$ and $\f_2$ can be represented by polyhedral sets. To further motivate the study of COPIC, let us consider the following illustration.

A spanning tree of a graph needs to be constructed as a backbone network. To
construct a link of the tree, many different tasks need to be completed, such as
digging, building conduits, laying fiber cables, lighting dark fiber etc. Each
of the tasks needs to be  assigned to different contractors and for each link in
a graph the costs vary by quotes from different contractors. We want to assign
the tasks to contractors and choose an appropriate tree topology so that the
overall construction cost is minimized. This optimization problem can be
formulated as a COPIC where feasible solution sets $\f_1$ and $\f_2$ correspond to
spanning trees and assignments of tasks to contractors, respectively.
\medskip

In this paper we investigate various theoretical properties of COPIC. To
understand the impact of interaction costs in combinatorial optimization we will
analyze special cases of the interaction cost matrix $Q$ for representative
well-studied sets of feasible solutions. Among others, the classes of
interaction cost matrices $Q$ that we will be focused on in this paper include
matrices of fixed rank, and diagonal matrices. 
In the literature many quadratic-like optimization problems have been investigated in the context of fixed rank or low rank cost matrices, for example see~\cite{Allemand2001,Bouras96,punnen2015bipartite,YK95}. Further, the
importance of investigating COPIC with diagonal matrices is illustrated by its
direct connections to problems of disjointness of combinatorial
structures~\cite{roskind1985note, gabow1992forests, vygen1994disjoint,
frank1988packing}, packing, covering and partitioning
problems~\cite{bernath2015tractability}, as well as to problems of congestion
games~\cite{ackermann2008impact, werneck2000finding}. In this paper we also pose
the problem of identifying cost structures  of COPIC instances that can be
reduced to an instance with no interaction costs. These instances are
called \emph{linearizable instances}~\cite{CDW16, KP11, PK13, CP15, CSPB16}. We
suggest an approach of identifying such instances for COPIC with specific feasible solution structures along with a characterization of linearizable instances.

The aforementioned topics are investigated on COPIC's with representative well-studied sets of feasible solutions $\f_1$, $\f_2$. 
To make easy future references to different sets of feasible solutions we introduce shorthand notations. We denote by $2^{[n]} = \{S \colon S \subseteq [n]\}$ the unconstrained
solution set. Given a matroid $\M$ we denote by $\base(\M)$ the set of bases of
$\M$. We denote by $\umat{n}{k}$ the uniform matroid, whose base set
$\base(\umat{n}{k})$ is the set of all $k$-sets of $[n]$. Given a graph $G$,
$\M(G)$ is the graphic matroid of $G$, whose base set $\base(\M(G))$ is the set of all spanning
trees of $G$ (or spanning forests if $G$ is not connected). The set of all maximum
matchings of $G$ is denoted by $\pmatch(G)$. Given two terminals $s,t \in V(G)$
the set of all $s$-$t$-paths in $G$ is denoted by $\paths_{s,t}(G)$. If $G$ is a
directed graph $\paths_{s,t}(G)$ is the set of all directed $s$-$t$-paths in
$G$. The set of all cuts in $G$ is denoted by $\cuts(G)$ and $\cuts_{s,t}(G)$ is
the set of all $s$-$t$-cuts in $G$.

Using these definitions, for example,  the bipartite unconstrained quadratic programming
problem~\cite{punnen2015bipartite} is denoted by $\copic(2^{[m]}, 2^{[n]}, Q, c, d)$.
\medskip

The structure of this paper is as follows. We begin by discussing the complexity
of COPIC with no significant constraints on the cost structure in
Section~\ref{sec:complexity}. Section~\ref{sec:rank} investigates the case when
the interaction cost matrix $Q$ is of fixed rank. Using the methods from parametric
optimization we show that in the case when one of the solution sets is
unconstrained, i.e.\@ $\f_1=2^{[n]}$ or $\f_2=2^{[m]}$, and linear cost optimization over the other
solution set can be done in polynomial time, the problem becomes polynomially
solvable. Further, we show that approximability may be achieved in the
case of $Q$ with fixed rank. We also show that if the number of breakpoints of
multi-parametric linear optimization over both sets of feasible solutions is
polynomially bounded and if $Q$ has fixed rank, then COPIC can be solved in polynomial time. Section~\ref{sec:diagonal} investigates COPIC's
where interaction cost matrix $Q$ is diagonal. That is, there is a one-to-one
relation between ground elements of $\f_1$ and $\f_2$ and the interaction costs
appear only between the pairs of the relation. The complexity of COPIC with
various well-knows feasible structures (matroids, paths, matchings, cuts, etc.)
in the context of diagonal matrix $Q$ are considered, and their relationship to
some existing results in the literature is presented. Characterization of linearizable instances is
investigated in Section~\ref{sec:lin}. The paper is concluded with
Section~\ref{sec:conclusion}, where we summarize the results and suggest some
problems for future work.

\section{General complexity}\label{sec:complexity}

Being a generalization of many hard combinatorial optimization problems, the general
COPIC is NP-hard. Moreover, even for the ``simple" case with no constraints on
the feasible solutions it results in the bipartite unconstrained quadratic programming
problem which is NP-hard \cite{punnen2015bipartite}. 
$\copic(2^{[m]}, 2^{[n]}, Q, c, d)$ can easily be embedded into a COPIC for most sets of feasible solutions $\f_1$ and $\f_2$, which implies again NP-hardness.
However, $\copic(2^{[m]}, 2^{[n]}, Q, c, d)$ is known to be solvable in polynomial time if
$Q \leq 0$ and if $Q,c,d \geq 0$ (see Punnen et al.~\cite{punnen2015bipartite}).
This is not true anymore if $\f_{1}, \f_{2}$ are bases of a uniform
matroid, for which we obtain the following hardness result.

\begin{theorem}\label{compumat}
    $\copic(\base(\umat{m}{k_1}), \base(\umat{n}{k_2}), Q, 0, 0)$ is
   strongly NP-hard even if $Q \geq 0$.
\end{theorem}
\begin{proof}
We give a reduction from a strongly NP-hard version of the cardinality constrained directed minimum cut problem.


Let $\vec{K}_{m,n}$ be a digraph with vertex sets $[m]$ and $[n]$ and arcs
$(i,j)$ for each $i \in [m]$ and $j \in [n]$. The \emph{k-card min directed cut
problem} asks for a minimum cost directed cut $\delta^{+}(S) = \{ (i,j) \colon i \in S, j
\notin S \}$ such that $|\delta^{+}(S)| = k$.
Using similar
arguments as in~\cite{bruglieri2004cardinality} one can show that this directed version of the minimum cut problem is strongly NP-hard.
Now we show how this problem can be solved in polynomial time,
assuming a polynomial time algorithm for $\copic(\base(\umat{m}{k_1}),
\base(\umat{n}{k_2}), Q, 0, 0)$ exists. 

For each $k_{1} = 1,2, \dots, m$ check if $\frac{k}{k_{1}}$ is an integer. If so set
$k_{2} = \frac{k}{k_{1}}$ and solve the instance $\copic(\base(\umat{m}{k_{1}}),
\base(\umat{n}{k_{2}}), Q, 0, 0)$, obtaining solution sets $S_{1}, S_{2}$.
Note that $|S_{1}| |S_{2}| = k$, i.e. it corresponds to exactly $k$ edges.
We can define an equivalent directed cut $\delta^{+}(S)$ by setting 
\[ S = S_{1} \cup ([n] \setminus S_{2}). \]
This way the directed cuts $\delta^{+}(S)$ are in one to one correspondence with
solutions of COPIC. Doing this for all possible pairs $(k_1, k_2)$, we can
obtain all possible $k$-cuts as feasible solutions of instances of
$\copic(\base(\umat{m}{k_{1}}), \base(\umat{n}{k_{2}}), Q, 0, 0)$.
Taking the minimum found via all such COPIC problems solves the $k$-card directed
min cut problem in the given bipartite digraph. 
\end{proof}

Theorem~\ref{compumat} can be used to show that $\copic(\f_{1}, \f_{2}, Q, 0,
0)$ is NP-hard already for $Q \geq 0$ for most sets of feasible solutions
$\f_{1}, \f_{2}$, since in many cases cardinality constraints can be easily
encoded in more complicated sets of feasible solutions.


\smallskip
On the positive side, if we fix one of the two solutions, e.g.\@ $S_{1} \in \f_{1}$,
then finding the corresponding optimal solution $S_{2} \in \f_{2}$ reduces to
solving $\lcop(\f_{2}, h)$, where 
\begin{equation}\label{lcop_costs}
	h_{j}:= \sum_{i \in S_{1}} q_{ij} + d_{j} \ \ \ \text{ for } j \in [n].
\end{equation}
This implies that if the cardinality of one set of feasible solutions, say $\f_1$, is
polynomially bounded in the size of the input, then we can solve COPIC by solving
linear instances $\lcop(\f_2,h)$ (where $h$ is defined by \eqref{lcop_costs}) for all $S_1\in\f_1$.

\begin{theorem}
    If $m = O(\log n)$ and $\lcop(\f_{2}, h)$ can
    be solved in polynomial time for any cost vector $h \in \r^{n}$,
    then $\copic(\f_{1}, \f_{2}, Q, c, d)$ can be solved in polynomial time.
\end{theorem}

\section{The interaction matrix with fixed rank}\label{sec:rank}

In this section we investigate the behavior of COPIC in terms of complexity and approximability when the rank of the interaction costs matrix $Q$ is fixed. 
In the literature, many optimization problems have been investigated in the context of fixed rank or low rank cost matrices. This also includes problems with quadratic-like objective functions. For example, the Koopmans-Beckmann QAP~\cite{Bouras96}, the unconstrained zero-one quadratic maximization problem~\cite{Allemand2001}, bilinear programming problems~\cite{YK95}, the bipartite unconstrained quadratic programming problem~\cite{punnen2015bipartite}, among others.

Let $\rank(Q)$ denotes the rank of a matrix $Q$. Then $\rank(Q)$ is at most $r$, if and only if there exist vectors $a_{p}
= (a^{(p)}_{1}, a^{(p)}_{2}, \dots, a^{(p)}_{m}) \in \r^{m}$
and $b_{p} = (b^{(p)}_{1}, b^{(p)}_{2}, \dots, b^{(p)}_{n}) \in \r^{n}$ for $p=1,2,\ldots,r$, such that
\begin{equation}\label{fact}
    Q = \sum_{p=1}^r a_{p} \transp{b_{p}}.
\end{equation}
We say that \eqref{fact} is a \emph{factored form} of $Q$. Then
$\copic(\f_{1}, \f_{2}, Q,c,d)$, where $Q$ is of fixed rank $r$, becomes minimizing
\begin{equation}\label{COPICfixed}
    f(S_1,S_2)=\sum_{p=1}^r \left(\sum_{i\in S_1} a_{i}^{(p)} \sum_{j\in S_2}
    b_{j}^{(p)} \right) + \sum_{i\in S_1} c_{i} + \sum_{j\in S_2} d_{j},
\end{equation}
such that $S_1\in\f_1$, $S_2\in\f_2$.

In the following, we show that if $\f_{1} (= 2^{[m]})$ is unrestricted, i.e.\@ the set
of all subsets of $[m]$, then we can generalize the results of
Punnen~et~al.~\cite{punnen2015bipartite} to solve the problem. Using methods of
multi-parametric optimization we also demonstrate how to tackle more-general
problems where both sets of feasible solutions are constrained, if their
parametric complexity is bounded.

These results are obtained using methods from binary and linear optimization.
To apply these techniques we will formulate our problem in terms of binary
variables. We achieve this in a straightforward way, by introducing variables $x \in
\{0,1\}^{m}, y \in \{0,1\}^{n}$ in one to one correspondence with a solution $S_{1}, S_{2}$,
such that $x_{i} = 1$ iff $i \in S_{1}$, and $y_{j} = 1$ iff $j \in S_{2}$. 
The vector $x$ and $y$ are respectively called the incidence vectors of $S_1$ and $S_2$.
Thus the family of feasible solutions can be represented in terms of the incidence vectors, i.e.\@
 $\f'_{1} = \{ x \in \{0,1\}^{m} \colon S_{1} \in \f_{1}\text{ and } (x_{j} = 1 \Leftrightarrow j \in
S_{1}) \}$ and $\f'_{2} = \{ y \in \{0,1\}^{n} \colon S_{2} \in \f_{2} \text{
and } (y_{j} = 1 \Leftrightarrow j \in S_{2}) \} $.
Now, rank $r$ COPIC can be formulated as the binary optimization problem:
\begin{align*}
    \min\ & \sum_{p=1}^{r} (\transp{a_{p}} x) (\transp{b_{p}} y) + \transp{c} x +
    \transp{d} y\\
     \text{s.t.\ \ } & x \in \f'_{1}\\
     & y \in \f'_{2}
\end{align*}

\subsection{One-sided unconstrained fixed rank COPIC}

In this section we consider the case where $\f'_{1} = \{0,1\}^{m}$. Observe that
COPIC is equivalent to
the following linear relaxation of the constraint $x \in \{0,1\}^{m}$.
\begin{align*}
    \min\  & \sum_{p=1}^{r} (\transp{a_{p}} x) (\transp{b_{p}} y) + \transp{c} x +
    \transp{d} y\\
     \text{s.t.\ \ } & x \in [0,1]^{m}\\
     & y \in \f'_{2}
\end{align*}
To solve this problem, consider the multi-parametric linear program (MLP)
\begin{align*}
    h_{1}(\lambda) := \min\ & c^T x\\
    \text{s.t.\ } & \transp{a_{p}} x = \lambda_{p} \quad \text{for }p = 1,2,\dots,r\\
    & x \in [0,1]^{m},
\end{align*}
where $\lambda = (\lambda_{1}, \lambda_{2}, \dots, \lambda_{r}) \in \r^{r}$. Then
$h_{1}(\lambda)$ is a piecewise linear convex function~\cite{gal1995param}.
A basis structure for MLP is a partition $(\base, \mathcal{L}, \mathcal{U})$ of
$[m]$, such that $|\base| = r$. With each basic feasible solution of MLP we
associate a basis structure $(\base, \mathcal{L}, \mathcal{U})$, where
$\mathcal{L}$ is the index set of nonbasic variables at the lower bound $0$,
$\mathcal{U}$ is the index set of nonbasic variables at the upper bound $1$ and
$\base$ is the index set of basic variables.
Given a dual feasible basis structure $(\base, \mathcal{L}, \mathcal{U})$, the
set of values $\lambda \in \r^{r}$ for which the corresponding basic solution is
optimal is called the characteristic region of $(\base, \mathcal{L},
\mathcal{U})$.
Since $h_{1}(\lambda)$ is piecewise linear convex, $h_{1}(\lambda)$ is linear if $\lambda$ is
restricted to a characteristic region associated with a dual feasible basic
structure $(\base, \mathcal{L}, \mathcal{U})$. We call the extreme points of the
characteristic regions of $(\base, \mathcal{L}, \mathcal{U})$ as breakpoints and
denote the set of these breakpoints by $B_{1}$ and define $x(\lambda)$
as the optimal basic feasible solution of $h_{1}(\lambda)$ at each $\lambda \in
B_{1}$.
By the results of Punnen et al.~\cite[Theorem~3]{punnen2015bipartite} we
know that $x(\lambda) \in \{0,1\}^{m}$. Let $y(\lambda) \in \f_{2}$ be an
optimal solution to our instance of COPIC when $x$ is fixed at $x(\lambda)$.
In this case COPIC reduces to
\begin{align*}
    \min\ & \left(\sum_{p=1}^{r} (\transp{a_{p}} x(\lambda)) \transp{b_{p}}  +
    \transp{d}\right) y\\
     \text{s.t.\ \ } & y \in \f_{2}
\end{align*}
which is an instance of $\lcop(\f_{2}, f)$, with $f = \sum_{p=1}^{r}
(\transp{a_{p}} x(\lambda)) b_{p} + d$. This allows us to calculate $y(\lambda)$
in $O(\max\{r m n, T(\f_{2}) \})$ time, using an $T(\f_{2})$-time algorithm for
$\lcop(\f_{2}, f)$, for each $\lambda \in B_{1}$.

\begin{theorem}\label{thm:fixrank-paramsol}
    There exists an optimal solution to $\copic(2^{[m]}, \f_{2}, Q, c, d)$ with
    $\rank(Q) = r$ amongst the solutions $\{(x(\lambda), y(\lambda)) \colon
        \lambda \in B_{1}
\}$.
\end{theorem}
\begin{proof}
    Rank $r$ COPIC is equivalent to solving the bilinear program
    \begin{align*}
        \min\ & \sum_{p=1}^{r} \lambda_{p} (\transp{b_{p}} y) + \transp{c} x +
        \transp{d} y\\
        \text{s.t.\ \ } & \transp{a_{p}} x = \lambda_{p} \quad p = 1, 2, \dots, r\\
        & x \in [0,1]^{m}, y \in \f_{2}, \lambda \in \r^{r}.
    \end{align*}
    Let $h(\lambda)$ be the optimal value if $\lambda$ is fixed, then we can
    decompose $h(\lambda)$ into $h(\lambda) = h_{1}(\lambda) + h_{2}(\lambda)$,
    where
    \begin{align*}
        h_{2}(\lambda) = \min\ & \sum_{p=1}^{r} \lambda_{p} (\transp{b_{p}} y) +
        \transp{d} y \\
       \text{s.t.\ \ } & y \in \f_{2}.
    \end{align*}
    So rank $r$ COPIC can be reduced to solving
    \[ \min_{\lambda \in \r^{r}} h(\lambda). \]
    We already argued above that $h_{1}(\lambda)$ is a piecewise linear convex 
    function in $\lambda$. Using the fact that $h_{2}(\lambda)$ is the pointwise minimum of
    linear functions, we obtain that $h_{2}(\lambda)$ is a piecewise linear
    concave function in $\lambda$~\cite{boyd2004convex}. This implies that
    $h_{1}(\lambda)$ is linear, if $\lambda$ is restricted to any
    characteristic region of $h_{1}(\lambda)$ and thus $h(\lambda)$ is concave on each of
    these regions. This implies that the minimum of $h(\lambda)$ is attained at
    a breakpoint of $h_{1}(\lambda)$, which implies the result since $B_{1}$ is
    defined as the set of these breakpoints.
\end{proof}

Analogously to Punnen et al.~\cite{punnen2015bipartite}, we can use
Theorem~\ref{thm:fixrank-paramsol} to solve rank $r$ COPIC using the
following approach.

\begin{enumerate}
    \item Compute the set $\bar{S}$ of all optimal basic feasible solutions
        corresponding to the extreme points of the characteristic region of a
        dual  feasible basis structure $(\base, \mathcal{L},
        \mathcal{U})$ of $h_{1}(\lambda)$.
    \item For each $x \in \bar{S}$ compute the best $y \in \f_{2}$ by solving
        $\lcop(\f_{2}, f)$, with $f = \sum_{p=1}^{r} (\transp{a_{p}} x) b_{p} + d$.
    \item Output the best pair $(x,y)$ with minimum total cost found in the last
        step.
\end{enumerate}

By the arguments above it follows that this algorithm finds an optimal
solution. There are $\binom{m}{r}$ choices for $\base$ and each of them gives a
unique allocation of non-basic variables to $\mathcal{L}$ and $\mathcal{U}$
(uniqueness following from non-degeneracy which can be achieved by appropriate
perturbation of the cost vector). The basis inverse can be obtained in
$O(r^{3})$ time and given this inverse $\mathcal{L}$ and $\mathcal{U}$ can be
identified in $O(m r^{3})$ time, such that $(\base, \mathcal{L}, \mathcal{U})$
is dual feasible. This implies that the set of dual feasible basis structures is
bounded by $\binom{m}{r}$ and can be calculated in $O(\binom{m}{r} (r^{3} + m
r^{2}))$ time. By~\cite[Theorem 3]{punnen2015bipartite}, we know that the number
of extreme points associated with $(\base, \mathcal{L}, \mathcal{U})$ is bounded
by $2^{r}$ and how to calculate the optimal solution of $h_{1}(\lambda)$ for
$\lambda$ fixed at these extreme points without explicitly calculating
$\lambda$. This allows us to compute $\bar{S}$ in $O(\binom{m}{r} 2^{r} m)$
time. Fixing $x \in \bar{S}$, the best corresponding solution $y$ can be computed
in $O(\max\{m r n, \runtime(\f_{2})\})$ time. Summarizing this gives the following result.

\begin{theorem}
    If $\rank(Q) = r$ and there is a $\runtime(\f_{2})$-time algorithm for
    $\lcop(\f_{2}, f)$ for every $f \in \r^{n}$, then $\copic(2^{[m]}, \f_{2}, Q, c,
    d)$ can be solved in $O(\binom{m}{r} 2^{r} \max\{m r n, \runtime(\f_{2})\})$ time.
\end{theorem}
\begin{rem}
    An identical approach works for sets of feasible solutions $\f_{1}$, for which we can
    solve the linear cost minimization problem, extended by a constant number of
    side constraints of the form $\transp{a_{p}} x = \lambda_{p}$ and the number of breakpoints (in $\lambda$) is polynomially
    bounded. But this does not help for most non-continuous problems, because
    already for the bases of a uniform matroid this corresponds to a partition problem. 
\end{rem}
We can now use Theorem~\ref{thm:fixrank-paramsol} to obtain 
approximation algorithms for rank $r$ COPIC based on approximation algorithms
for the linear problem with feasible solutions in $\f_{2}$.

\begin{theorem}\label{thm:fixrank-approx}
    $\copic(2^{[m]}, \f_{2}, Q, c, d)$ such that $\lcop(\f_{2}, f)$ admits a $\runtime(\f_{2})$ time
    $\alpha$-approximation algorithm for arbitrary $f \in \r^{n}$, has a 
        $O(\binom{m}{r} 2^{r} \max\{m r n, \runtime(\f_{2})\})$ time 
        $\alpha$-approximation algorithm.
\end{theorem}
\begin{proof}
    By Theorem~\ref{thm:fixrank-paramsol} there exists an optimal solution
    \[ (x^{*}, y^{*}) = (x(\lambda^{*}), y(\lambda^{*})) \in \{(x(\lambda),
    y(\lambda)) \colon \lambda \in B_{1}\}.\] By the method above
    we will in some iteration find $x^{*}$ as one of the extreme points of a
    characteristic region of $h_{1}(\lambda)$. Then calculating $y^{*}$ is
    equivalent to solving $\lcop(\f_{2}, f)$ with $f = \sum_{p=1}^{r}
    (\transp{a_{p}} x^{*}) b_{p} + d$.
    Instead of solving this problem to optimality we can use our
    $\alpha$-approximation algorithm and obtain a solution $\tilde{y} \in
    \f_{2}$ such that 
    \[ \tilde{h}_{2} := \sum_{p=1}^{r} (\transp{a_{p}} x^{*}) (\transp{b_{p}}
        \tilde{y}) + \transp{d} \tilde{y}
    \leq \alpha h_{2}(\lambda^{*}). \]
    Altogether for our found solution $(x^{*}, \tilde{y})$ we obtain a bound on
    the objective value given by
    \[ h_{1}(\lambda^{*}) + \tilde{h}_{2} \leq h_{1}(\lambda^{*}) + \alpha
    h_{2}(\lambda^{*}) \leq \alpha h(\lambda^{*}). \]    
\end{proof}

For the more general case of rank $r$ COPIC, where both $\f_{1}$ and
$\f_{2}$ are constrained, we can still obtain a FPTAS based on the results of
Mittal and Schulz~\cite{mittal2013fptas}, for a restricted class of objective
functions.

\begin{theorem}[Mittal and Schulz~\cite{mittal2013fptas}]
    Consider the separable bi-linear programming problem
	\begin{align*}
        \min\ & \sum_{p=1}^{r} (\transp{a_{p}} x) (\transp{b_{p}} y) + \transp{c} x + \transp{d} y\\
         \text{s.t.\ \ } & x \in P_{1}\\
         & y \in P_{2}
    \end{align*}
    where $P_{1}, P_{2}$ are polytopes, completely given in terms of linear
    inequalities or by a polynomial time separation oracle, for fixed $r$. Then
    the problem admits a FPTAS giving a solution that is an extreme point of
    $P_{1}, P_{2}$, if $\transp{c} x > 0, \transp{d} y > 0$ and
    $\transp{a_{p}} x > 0, \transp{b_{p}} y > 0$ for $p=1,2,\dots,r$ over
    the polytopes $P_{1}, P_{2}$.
\end{theorem}

This result directly implies a FPTAS for $\copic(\f_{1}, \f_{2}, Q, c, d)$, if
$\rank(Q) = r$ and the sets $\f_{1},
\f_{2}$ can be represented as polytopes of polynomial size or polytopes with a polynomial
time separation oracle. This is for instance the case for matroid constraints.
See \cite{mittal2013fptas} for a detailed description of the FPTAS.

\subsection{General fixed rank COPIC via multi-parametric optimization}

To solve fixed rank COPIC when both sets of feasible solutions $\f_{1}$ and
$\f_{2}$ are constrained, we again apply methods from parametric optimization.
Since in many cases additional linear constraints of the form $\transp{a_{p}} x
= \lambda_{p}$ imply NP-hardness, we cannot follow an identical approach as
above. Instead, we analyze and solve multi-parametric objective versions for both sets of
feasible solutions directly. Given  linear cost vectors $a_{1}, a_{2}, \dots
a_{r} \in \r^{n}$ and $c \in \r^{n}$ in addition to a set of feasible solutions
$\f \subseteq \{0,1\}^{n}$, the problem of finding optimal solutions to
\begin{align*}
    \min\   & \sum_{p=1}^{r} \mu_p ( \transp{a_{p}} x) +
    \transp{c} x \\
    \text{s.t. } & x \in \f 
\end{align*}
for all possible values of $\mu \in \r^{r}$ is called multi-parametric linear
optimization over $\f$. In this section the number of vectors $a$ will always be
fixed to $r$. For every fixed $\mu \in \r^{r}$ this is equivalent to
solving an instance of $\lcop(\f, h)$ for $h = \sum_{p=1}^{r} \mu_{p} a_{p} +
c$. We denote this problem by $\mplcop(\f, a, c)(\mu)$.

It is well known that $\mplcop(\f, a, c)(\mu)$ is a piecewise-linear
concave function in $\mu$ on $\r^{r}$. For such a function the parameter space $\r^{r}$
can be partitioned into regions $M_{1}, M_{2}, \dots, M_{l}$, such that in each of
these regions the optimal objective value is linear in $\mu$ and for each $i =
1,2, \dots, l$ there exists a solution $x_{i} \in \f$ that achieves this value on the
whole region $M_{i}$. The smallest needed number $l$ of such regions is called the
parametric complexity of $\mplcop(\f, a, c)$.
B\"{o}kler and Mutzel~\cite{bokler2015output} showed that there is an
output-sensitive algorithm for $\mplcop(\f, a, c)$ to obtain all the 
solutions $x_{1}, x_{2}, \dots, x_{l}$ with running time
$O(\operatorname{poly}(n, m, l^{r}))$, if $\lcop(\f, h)$ can be solved in
polynomial time.

\smallskip

Given an instance of fixed rank COPIC
\begin{align*}
    \min\ & \sum_{p=1}^{r} (\transp{a_{p}} x) (\transp{b_{p}} y) + \transp{c} x +
    \transp{d} y\\
     \text{s.t.\ \ } & x \in \f'_{1}\\
     & y \in \f'_{2}
\end{align*}
and its optimal solution $(x^{*}, y^{*}) \in \f'_{1} \times \f'_{2}$, we observe 
that $x^{*}$ is an optimal solution to $\mplcop(\f'_{1}, a, c)(\mu^{*})$ for
$\mu_{p}^{*} = \transp{b_{p}} y^{*}$ and $y^{*}$ is an optimal solution to
$\mplcop(\f'_{2}, b, d)(\lambda^{*})$ for $\lambda_{p}^{*} = \transp{a_{p}}
x^{*}$. This yields the following approach for solving such instances of COPIC:

\begin{enumerate}
    \item Obtain optimal solutions $x_{1}, x_{2}, \dots, x_{l_{1}}$ for all
        possible parameter values $\mu$ of
        
        \noindent $\mplcop(\f'_{1}, a, c)(\mu)$ and $y_{1},
        y_{2}, \dots, y_{l_{2}}$ for all possible parameter values $\lambda$ of
        
        \noindent $\mplcop(\f'_{2}, b, d)$.
    \item Calculate their corresponding parameter values
        $\lambda^{(1)}, \lambda^{(2)}, \dots, \lambda^{(l_{1})}$ and $\mu^{(1)},
        \mu^{(2)}, \dots, \mu^{(l_{1})}$ as $\lambda^{(i)}_{p} = \transp{a_{p}} x_{i}$
        and $\mu^{(j)}_{p} = \transp{b_{p}} y_{j}$.
    \item For each pair $(x_{i}, y_{j})$ 
        check if $x_{i}$ is optimal for $\lcop(\f'_{1}, a, c)(\mu^{(j)})$ and $y_{j}$ is
        optimal for $\lcop(\f'_{2}, b, d)(\lambda^{(i)})$.
    \item Among all the pairs that fulfill conditions in (3), take the one with minimum objective value
        for our instance of COPIC.
\end{enumerate}
To guarantee that this method finds the optimal solution $(x^{*}, y^{*})$ the two given instances of
$\mplcop$ must be non-degenerate. This can be guaranteed by appropriate
perturbations of the cost vectors. Based on the algorithm of B\"{o}kler and
Mutzel~\cite{bokler2015output} we obtain the following result.

\begin{theorem}
    Let $l_{1}, l_{2}$ be the parametric complexity of $\mplcop(\f'_{1}, a, c),
    \mplcop(\f'_{2}, b, d)$ respectively, and
    $\rank(Q) = r$ is a constant. If both
    $\lcop(\f_{1}, h)$ and $\lcop(\f_{2}, h)$ can be solved in polynomial time
    for arbitrary linear cost vectors $h$,
    then $\copic(\f_{1}, \f_{2}, Q, c, d)$ can be solved in
    $O(\operatorname{poly}(n, m, l_{1}^{r}, l_{2}^{r}))$ time.
\end{theorem}

If $\f$ is the set of bases of a matroid, Ganley~et~al.~\cite{ganley1995multi}
showed that the parametric complexity of  $\mplcop(\f, a,
c)$ for arbitrary $a, c$ over the whole parameter region $\r^{r}$ is polynomially bounded, if $r$ is fixed.

\begin{theorem}[Ganley et al.~\cite{ganley1995multi}]\label{thmganley}
    If $\f = \base(\M)$ is the set of bases of a matroid $\M$ with $n$ elements,
    the parametric complexity of $\mplcop(\f, a, c)$ for arbitrary
    $a$ and $c$ is bounded by $O(n^{2r-2})$.
\end{theorem}

This implies a polynomial time algorithm for $\copic(\base(\M_{1}),
\base(\M_{2}), Q, c, d)$ for arbitrary matroids $\M_{1}, \M_{2}$ and fixed rank
matrix $Q$.
For rank 1 problems Eppstein~\cite{eppstein1998geometric} gives stronger bounds
for the parametric complexity. It is conjectured that for higher rank
matrices and bases of matroids as feasible solution,
stronger bounds than the one given in Theorem~\ref{thmganley} can be achieved.
For other types of feasible solutions, like
paths or bipartite matchings, this approach does not yield polynomial time
algorithms.

Another set of feasible solutions for which this approach yields a polynomial
time algorithm are global cuts in a graph. Karger~\cite{karger2016enumerating}
 currently gives the best bound for the parametric complexity of cuts and
 obtains several other related sets of feasible solutions with similar polynomial bounds.
In this case we are even able to bound the number of distinct cuts that can become optimal,
instead of just the parametric complexity, so degeneracy is not even an issue here.

\begin{theorem}[Karger~\cite{karger2016enumerating}]
    If $\f = \cuts(G)$ the number of cuts that can become optimal in $\mplcop(\f, a,
    c)$ for arbitrary $a$ and $c$ over all choices of $\mu$ is bounded by $O(n^{r+1})$.
\end{theorem}

Already for the case $r=1$ subexponential lower bounds for
the parametric complexity of $\mplcop(\f, a, c)$ for paths ($\f =
\paths_{s,t}(G)$) and matchings ($\f = \pmatch(G)$) in a graph $G$ are known
(Gusfield~\cite{gusfield1980sensitivity},
Carstensen~\cite{carstensen1984parametric}). However, in the
setting of smoothed
analysis Brunsch and R\"{o}glin~\cite{brunsch2015improved} showed that
the parametric complexity of $\mplcop(\f, a, c)$ is bounded by $O(n^{2r} \phi^{r})$ for every perturbation
parameter $\phi \geq 1$, all costs $a,c$ and arbitrary sets of feasible solutions $\f$.

\section{Diagonal interaction matrix}\label{sec:diagonal}

In this section we analyze the special case of COPIC, refereed to as \emph{diagonal COPIC}, where for a given vector
$a \in \r^{n}$ the matrix $Q = (q_{ij})$ is given as the diagonal $n\times n$ matrix 
\[ q_{ij} = \begin{cases}
        a_{i} & \text{if } i=j\\
        0 & \text{otherwise}.
\end{cases} \]
This results in finding solutions $S_{1} \in \f_{1} \subseteq \{0,1\}^{n}$ and $S_{2} \in \f_{2} \subseteq \{0,1\}^{n}$ that
minimize the objective function
\[ f(S_1,S_2) = \sum_{i\in S_1 \cap S_{2}} a_{i} + \sum_{i\in S_1}c_i +
\sum_{j\in S_2}d_j. \]
Such instances are denoted by $\copic(\f_{1}, \f_{2}, \diag(a), c, d)$.

Already this very restricted version of COPIC includes many well-studied
problems of combinatorial optimization. For example, problems that ask for two
disjoint combinatorial structures among an element set can all be handled by
solving COPIC with identity interaction matrix $Q = I$ and $c=d=0$. This includes the disjoint spanning tree problem~\cite{roskind1985note}, disjoint matroid base
problem~\cite{gabow1992forests}, disjoint path
problems~\cite{vygen1994disjoint,frank1988packing}, disjoint matchings problem~\cite{Frieze83}
and many others. Bernáth and Király~\cite{bernath2015tractability} analyzed the
computational complexity of many
combinations of different packing, covering and partitioning problems on graphs
and matroids. It is easy to model all of these problems as instances of diagonal
COPIC. The hardness results for packing problems in this paper directly imply
NP-hardness results for diagonal COPIC with $Q=I$ and $c=d=0$ for several classes of
problems. 
In this section we further investigate complexity of diagonal COPIC.
Some results investigated in this section are summarized in Table~\ref{tab:diagonal-summary}.

\begin{table}[h]
    \begin{tabularx}{\textwidth}{l | X X X X X }
        \toprule
        $\f_{1}\ \setminus\ \f_{2}$  & $2^{[n]}$ & $\base(\umat{n}{k_{2}})$ &
        $\base(\M_{2})$ & $\pmatch(G)$ & $\paths_{s_{2}, t_{2}}(G)$  \\
        \midrule
        $2^{[n]}$ & $O(n)$  & P & P & P & P \\
        $\base(\umat{n}{k_{1}})$ &  & P  & P
        \tiny{($c=d=0$)} & open & open \\
        $\base(\M_{1})$ & & & P \tiny{($c=d=0$)}
         & open & NP-hard \\
        $\pmatch(G)$ &  & & & NP-hard \cite{Frieze83} & open \\
        $\paths_{s_{1}, t_{1}}(G)$ &  & & & & NP-hard  \\
        \midrule
    \end{tabularx}
	\smallskip
    \caption{Summary of complexity results for COPIC with a diagonal matrix}    
	\label{tab:diagonal-summary}
\end{table}

\subsection{Unconstrained feasible sets}
\label{sec:diagonal-unconstrained}

We start by considering diagonal COPIC with unconstrained feasible sets.
\begin{theorem}\label{uncondiag}
    $\copic(2^{[n]}, 2^{[n]}, \diag(a), c, d)$ can be
    solved in linear time.
\end{theorem}
\begin{proof}
    For each $e \in [n]$ independently we have four different choices:
    \begin{itemize}
        \item $e \notin S_{1} \cup S_{2}$: this contributes $0$ to $f(S_{1}, S_{2})$
        \item $e \in S_{1}, i \notin S_{2}$: this contributes $c_{e}$ to
            $f(S_{1}, S_{2})$
        \item $e \notin S_{1}, i \in S_{2}$: this contributes $d_{e}$ to
            $f(S_{1}, S_{2})$
        \item $e \in S_{1} \cap S_{2}$: this contributes $a_{e} + c_{e} + d_{e}$
            to $f(S_{1}, S_{2})$
    \end{itemize}
    So for each $e \in [n]$ we can independently find $\min\{ 0, c_{e}, d_{e},
    a_{e} + c_{e} + d_{e} \}$ and select the corresponding solution accordingly. This can be done in
    constant time for each $e \in [n]$, so the overall running time is $O(n)$.
\end{proof}

The result of Theorem~\ref{uncondiag} can be generalized. Using a
straightforward dynamic programming approach, $\copic(\f_1,\f_2,Q,c,d)$ with matrix $Q$ of bandwidth $O(\log n)$ can be solved in polynomial time. This result is presented as Theorem~2.11 in the PhD thesis of Sripratak~\cite{Sripratak}.

\begin{theorem}
    $\copic(\f, 2^{[n]}, \diag(a), c, d)$ can be solved by solving
    $\lcop(\f, f)$, where $f_{i} = \min\{c_{i} + d_{i} + a_{i}, c_{i}\} -
    \min\{d_{i}, 0\}$ for each $i \in [n]$.
\end{theorem}
\begin{proof}
    For each $i \in [n]$ we can determine independently if it should be
    included in $S_{2}$, given that it is included in $S_{1}$ or not. The
    cost for an element $i\in [n]$ is therefore uniquely determined as $f_{1}(i) =
    \min\{ c_{i} + d_{i} + a_{i}, c_{i} \}$, if $i \in S_{1}$ and as $f_{2}(i) = 
    \min\{ d_{i}, 0 \}$ if $i \notin S_{1}$. So we can determine an optimal solution $S_{1}$ by solving the
    minimization problem over $\f$ for the linear cost function $f_{1} -
    f_{2}$. The corresponding corresponding optimal $S_{2}$ can be easily
    obtained in $O(n)$ time.
\end{proof}

\subsection{Uniform and Partition Matroids}
\label{sec:diagonal-uniformmat}

In the following two subsections we investigate diagonal COPIC where $\f_1$ and
$\f_2$ correspond to bases of different types of matroids. For bases of uniform
and partition matroids, which are defined by standard cardinality constraints,
the main insight is that we can solve our problem in polynomial time using matching
algorithms.

Given a graph $G=(V,E)$ and a function $b \colon V \ra 2^{\NN}$ an edge set $M
\subseteq E$ is a $b$-factor, if $|M \cap \delta(v)| \in b(v)$ for each $v \in
V$. If $b(v) = \{k\}$ for some integer $k \in \NN$ we simply write $b(v) =
k$. Given an additional cost function $c \colon E \ra \r$ a minimum cost
$b$-factor can be found in polynomial time, if all the $b$-values $b(v)$ are
sequences of consecutive integers $[b_{1}; b_{2}] = \{b_{1}, b_{1}+1, \dots,
b_{2}\}$ (see~\cite[Section 10.2]{lovasz2009matching}).

\begin{theorem}\label{thm:uniformmatroid-matching}
    $\copic(\base(\umat{n}{k_{1}}), \base(\umat{n}{k_{2}}), \diag(a), c, d)$ can be  solved in polynomial time.
\end{theorem}
\begin{proof}
    We create an equivalent instance of the minimum cost $b$-factor problem on a graph
    $G$ (see Figure~\ref{fig:uniformmatroid-matching}).
    To achieve this, we introduce two special vertices $x$ and
    $y$ with $b(x) = k_{1}$ and $b(y) = k_{2}$ and another $3n$ vertices $i_{x}$, $i_{y}$ and
    $i_{m}$ for $i=1,2,\dots,n$, i.e., for each element of the ground set of the
two matroids. We set $b(i_{x}) = b(i_{y}) = 1$ and $b(i_{m}) = \{0, 1 \}$.
    The $k_{1}$ vertices matched with $x$ and $k_2$ vertices matched with $y$ correspond to the sets $S_{1}$    and $S_{2}$, respectively.

    We introduce edges $\{x, i_{x} \}$ with cost $c_{i} + \frac{a_{i}}{2}$
    and $\{y, i_{y} \}$ with cost $d_{i} + \frac{a_{i}}{2}$. We also connect $\{
        i_{x}, i_{y} \}$ with edges of cost $0$ and  $\{ i_{x},
        i_{m} \}$ and $\{ i_{x}, i_{m} \}$ both with cost $-\frac{a_{i}}{2}$.
    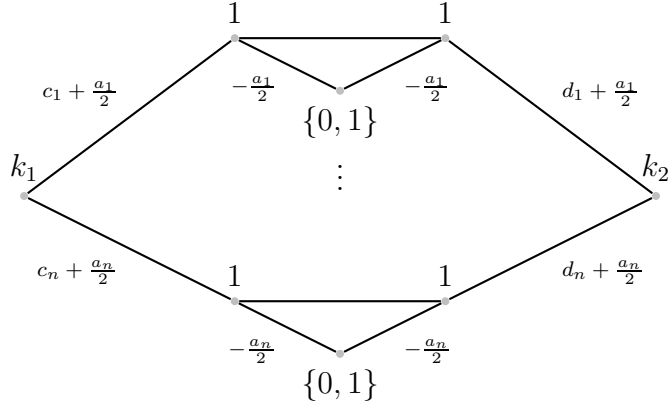
\begin{figure}[h]
        \begin{tikzpicture}[scale=0.7, auto,swap]

            \node[vertex,label=$k_{1}$] (x) at (0,-4) {};
            \node[vertex,label=$k_{2}$] (y) at (12,-4) {};

            \node[vertex,label=1] (x1) at (4,-1) {};
            \node[vertex,label=1] (y1) at (8,-1) {};
            \node[vertex,label=below:$\setzo$] (z1) at (6,-2) {};

            \node[vertex,label=1] (xn) at (4,-6) {};
            \node[vertex,label=1] (yn) at (8,-6) {};
            \node[vertex,label=below:$\setzo$] (zn) at (6,-7) {};

            \node (middle) at (6, -3.5) {\vdots};

            \path[edge] (x1) -- node[weight]{$c_{1} + \frac{a_{1}}{2}$} (x);
            \path[edge] (x) -- node[weight]{$c_{n} + \frac{a_{n}}{2}$} (xn);

            \path[edge] (y) -- node[weight]{$d_{1} + \frac{a_{1}}{2}$} (y1);
            \path[edge] (yn) -- node[weight]{$d_{n} + \frac{a_{n}}{2}$} (y);

            \path[edge] (x1) -- (y1);
            \path[edge] (x1) -- node[weight]{$-\frac{a_{1}}{2}$} (z1);
            \path[edge] (z1) -- node[weight]{$-\frac{a_{1}}{2}$} (y1);

            \path[edge] (xn) -- (yn);
            \path[edge] (xn) -- node[weight]{$-\frac{a_{n}}{2}$} (zn);
            \path[edge] (zn) -- node[weight]{$-\frac{a_{n}}{2}$} (yn);
        \end{tikzpicture}
        \caption{Illustration for the proof of
            Theorem~\ref{thm:uniformmatroid-matching}}
        \label{fig:uniformmatroid-matching}
    \end{figure}

    It is easy to see that there is a one-to-one mapping between feasible
    solutions of the given diagonal COPIC and this instance of the $b$-factor problem,  and moreover, the corresponding costs are the same.
    Any feasible $b$-factor $M$ must contain exactly $k_{1}$ edges of the form $\{x,
        i_{x}\}$ and $k_{2}$
        edges of the form $\{y, i_{y}\}$. These can be identified with the
        solution sets $S_{1}$ and $S_{2}$ for our diagonal COPIC. Given any
    such partial $b$-factor there exists exactly one completion to a feasible
    $b$-factor, using additional edges inside the triangles $i_{x},
    i_{y}, i_{m}$ for each $i \in [n]$, according to the following four cases.
    We can also directly
    observe that the cost of the enforced $b$-factor and the solution
    $S_{1}, S_{2}$ is the same.
    \begin{enumerate}
        \item $i \notin S_{1}, i \notin S_{2}$: Both $i_{x}$ and $i_{y}$ are 
            unmatched. The only way to match both is by using the single edge
            $\{i_{x}, i_{y}\}$ and leaving $i_{m}$ unmatched, which is feasible
            since $0 \in b(i_{m})$. The contribution to the total cost is $0$.
        \item $i \in S_{1}, i \notin S_{2}$: In this case $i_{x}$ is already
            matched but $i_{y}$ is still unmatched. The only feasible way to
            match $i_{y}$ is using the edge $\{i_{y}, i_{m}\}$, which
            contributes $c_{i}$ to the cost.
        \item $i \notin S_{1}, i \in S_{2}$: This case is symmetric to case (2). The cost
            contribution is $d_{i}$.
        \item $i \in S_{1}, i \in S_{2}$: In this case $i_{x}$ and $i_{y}$ are
            both already matched and $i_{m}$ cannot be matched anymore.
            We get a cost contribution of $a_{i} +
            c_{i} + d_{i}$.
    \end{enumerate}
\end{proof}

Given a partition $S_{1}, S_{2}, \dots, S_{t}$ of the ground set $E$ and
integers $g_{1}, g_{2}, \dots, g_{t}$, such that $0 \leq g_{i} \leq |B_{i}|$ for
all $i=1,2,\dots,t$, the set $\{ X \subseteq E \colon |X \cap S_{i}| = g_{i}
\text{ for all }i=1,2,\dots,t \}$ forms the collection of all bases of a partition matroid.

\begin{corollary}
    If $\M_1,\M_2$ are partition matroids, $\copic(\base(\M_{1}),
    \base(\M_{2}), \diag(a), c, d)$ can be solved in polynomial time.
\end{corollary}
\begin{proof}
    To prove this theorem, we use the approach based on matchings as in the proof of Theorem~\ref{thm:uniformmatroid-matching}, with minor modifications. Instead of special vertices $x$ and $y$, we introduce $x$- and
    $y$-vertices for each set in the partition of the ground set and connect
    these vertices only to the $i_{x}, i_{y}$-vertices that are in the
    corresponding set of the partition. The equivalence of this construction
    can be shown analogously.
\end{proof}

One can even further generalize the concept of partition matroids.
Given a partition $S_{1}, S_{2}, \dots ,S_{t}$ of the
element set $E$ and integers $f_{1}, f_{2}, \dots, f_{t}$,$ g_{1}, g_{2}, \dots,
g_{t}$, $k$, such that $0 \leq f_{i} \leq g_{i} \leq |S_{i}|$ for all $i=1,2,\dots,t$ and $\sum_{i=1}^{t}
f_{i} \leq k \leq  \sum_{i=1}^{t} g_{i}$.
The set $\{ X \subseteq E \colon |X| = k \text{ and } f_{i} \leq |X \cap S_{i}|
\leq g_{i}\, \forall i=1,2,\dots,t \}$ is the set of bases of a generalized
partition matroid~\cite{frank2011connections}.
A similar approach based on matchings still applies, if $\f_{1}, \f_{2}$ are
sets of bases of a generalized partition matroids. 

\medskip

The combination of bases of a uniform matroid with other sets of feasible
solutions in diagonal COPIC is similar to different versions of linear problems with
capacity side constraints. The following result is an example that can be
derived using methods for
the well-studied constrained shortest path problem with uniform edge weights, for
which dynamic programming can be used to solve the problem in polynomial time
(see Dumitrescu and Boland~\cite{dumitrescu2001algorithms} for a review).

\begin{theorem}
    $\copic(\base(\umat{m}{k}), \paths_{s,t}(G), \diag(a), 0, d)$ can be
    solved in polynomial time, if $a \geq 0$ and $d \geq 0$.
\end{theorem}

\subsection{Matroid bases as feasible sets}
\label{sec:diagonal-mat}

Another problem of great interest is the case when $\f_{1}, \f_{2}$ are sets of
spanning trees of a graph,
especially if the underlying graphs are isomorphic.  A generalization of this
problem is the case when $\f_{i} = \base(\M_{i})$ are the sets of bases of (not necessarily isomorphic)
matroids $\M_{1}, \M_{2}$. In this section we assume familiarity with matroids and refer the
reader to Oxley~\cite{oxley2006matroid} for further definitions, results and
notations.

We will first focus on the case without linear costs, i.e.\@ $c \equiv d \equiv 0$. So
the problem we are interested in is, given a ground set $E = [n]$ and a cost
vector $a \in \r^{n}$, to minimize the objective function
\[ f(B_{1}, B_{2}) = \sum_{i \in B_{1} \cap B_{2}} a_{i} \]
under the restrictions that $B_{1} \in \base(\M_{1}), B_{2} \in \base(\M_{2})$
for two given matroids $\M_{1}, \M_{2}$ over the ground set $E$.

\subsubsection{Minimum cardinality base intersection}

If the cost vector $a \equiv 1$ (together with $c\equiv d\equiv 0$) this gives the problem of minimizing the size of the intersection of the two matroid bases $B_{1}$ and $B_{2}$. It contains as a special case the disjoint matroid base problem for two given matroids, since there exist two disjoint bases if and only if the optimal solution has objective value $0$.

Gabow and Westermann~\cite{gabow1992forests} showed that the disjoint matroid base problem can be
efficiently solved under the assumption that there exist efficient oracles to
solve the static-base circuit problem. This means that for both matroids
$\M_{i}$, $i=1,2$,
independent set $S$ and element $e \notin S$, we can efficiently decide if $S
\cup \{e\}$ is independent in $\M_{i}$, and if not, output all elements in $C(e, S)$, the
unique cycle contained in $S \cup \{e\}$ of the matroid $\M_{i}$.

\subsubsection{Minimum cost base intersection}

The more general case, where for each element $i \in B_{1} \cap B_{2}$ we pay a
non-negative cost $a_{i} \geq 0$ was already studied in the algorithmic game
theory literature. It is equivalent to computing the socially optimal state of a
two player matroid congestion game. Ackermann et al.~\cite{ackermann2008impact} show that
this problem can be solved in polynomial time for an arbitrary number of players using the same approach that was used by Werneck et al.~\cite{werneck2000finding} to calculate the
socially optimal state in spanning tree congestion games.

To keep this work self contained we give a summary of their algorithm using the
notation of diagonal COPIC. We reduce the problem to an equivalent instance of
the minimum cost disjoint base problem, for which we can guarantee the existence
of two disjoint bases.

The idea of the construction is to double all elements of $E$.
The new ground set of elements is denoted by $E' = E_{1} \cup E_{2}$, where $E_{1}, E_{2}$
are two disjoint copies of the original ground set $E$. For $i \in E$ we write
$i_{1}$ for the copy of $i$ inside $E_{1}$ and $i_{2}$ for its copy in $E_{2}$.
We set $a_{i_{1}}= a_{i}$ and $a_{i_{2}}:= 0$ for each $i \in E$ 
and introduce two new matroids $\M'_{1}, \M'_{2}$, each with $E'$ as their ground set. 
The independent sets of
$\M'_{j}$ are all sets $S' \subseteq E'$ that do not contain both $i_{1}$ and
$i_{2}$ for any $i \in E$ and where $\{ i \in E \colon i_{1} \in S'\text{ or }
i_{2} \in S' \}$ is independent in $\M_{j}$, for $j=1,2$.

Given two disjoint bases $B'_{1}$ of $\M'_{1}$ and $B'_{2}$ of
$\M'_{2}$, they induce, not necessarily disjoint, bases $B_{1}, B_{2}$ of $\M_{1},\M_{2}$. For every element
$i \in B_{1} \cap B_{2}$ we know that both $i_{1}$ and $i_{2}$ were used in
$B'_{1}$ and $B'_{2}$. So for this element the cost $a_{i}$ is payed in the
disjoint base problem. For all other elements $i_{2}$ is used, since $0 =
a_{i_{2}} \leq a_{i_{1}}$.

Efficient methods for solving the minimum cost disjoint base
problem for  general matroids obtained by Gabow and
Westermann~\cite{gabow1992forests} can be used to solve our transformed minimum cost base intersection instance.
\medskip

The case of arbitrary real costs  $a_{e} \in \r$ can also be handled.
This is not included in the algorithmic game theory literature, since in that
context a positive impact of congestion (i.e.\@ $a_i<0$) does not make sense.

First, we find a set $B \in \inds(\M_{1}) \cap \inds(\M_{2})$ of minimum cost and
we contract this set. For all edges $e \in E \setminus B$ with $a_{e} < 0$ it
holds that $B+e \notin \inds(\M_{1}) \cap \inds(\M_{2})$, or we could improve
the solution, so these elements can never be in the intersection of a feasible
solution together with $B$. Hence we can run the algorithm from above on the 
remaining instance. The optimality of this approach follows from
the following lemma.

\begin{lemma}
    Let $B$ be an element of $\inds(\M_{1}) \cap \inds(\M_{2})$ with minimum cost $a(B):=\sum_{i\in B}a_i$, and $B_{1}, B_{2}$ be two bases. Then $B_{1}, B_{2}$ can be transformed into two new
    bases $\tilde{B}_{1}, \tilde{B}_{2}$ such that $B \subseteq \tilde{B}_{1}
    \cap \tilde{B}_{2}$ and $a(\tilde{B}_{1} \cap \tilde{B}_{2}) \leq a(B_{1} \cap B_{2})$.
\end{lemma}
\begin{proof}
    Let $e \in B \setminus (B_{1} \cap B_{2})$. There are three different cases
    on how to add $e$ to the intersection.
    \begin{enumerate}
        \item $e \notin B_{1} \cup B_{2}$: In this case we have $f_{i} \in
            C_{i}(e, B_{i}) \setminus B$ for both $i=1,2$. By modifying the
            bases to $\tilde{B}_{i} = B_{i} + e - f_{i}$ we get that
            \[ a(\tilde{B}_{1} \cap \tilde{B}_{2}) = a(B_{1} \cap B_{2}) + a_{e}
            - \begin{cases}a_{f} & f_{1} = f_{2}\\ 0 & f_{1} \neq f_{2}\end{cases} \]
        \item $e \in B_{1}, e \notin B_{2}$: In this case we have $f \in
            C_{2}(e, B_{2}) \setminus B$ and we can modify $\tilde{B}_{2} = B_{2} + e - f$.
            this gives a modified cost of 
            \[ a(\tilde{B}_{1} \cap \tilde{B}_{2}) = a(B_{1} \cap B_{2}) + a_{e}
            - \begin{cases}a_{f} & f \in B_{1}\\ 0 & f \notin B_{1}\end{cases} \]
        \item $e \notin B_{1}, e \in B_{2}$: symmetric to case (2).
    \end{enumerate}

    We apply these steps iteratively until $B$ is contained in the intersection.
    We know that the sum of costs of the elements $e \in \tilde{B}_{1} \cap
    \tilde{B}_{2}$ with $a_{e} \leq 0$ must now be smaller than before, since $B$
    is minimum. We never added any element $e$ to $\tilde{B}_{1} \cap
    \tilde{B}_{2}$ with $a_{e} > 0$. This implies that $a(\tilde{B}_{1} \cap
    \tilde{B}_{2}) \leq a(B_{1} \cap B_{2})$.
\end{proof}

The approach above gives us the following result.

\begin{theorem}
    $\copic(\base(\M_{1}), \base(\M_{2}), \diag(a), 0, 0)$ can be solved in
    polynomial time, for any two matroids $\M_{1}, \M_{2}$ and cost vector $a
    \in \r^{n}$.
\end{theorem}

\subsubsection{The case $a \geq 0, c\equiv d$}

This case can be solved analogously to the case without linear costs. We create
two identical helper matroids $\M'_{1}, \M'_{2}$, with the only difference that
we set the costs of the elements to $a_{e} + c_{e}$ and $c_{e}$. Since $a_{e} \geq 0$, it follows that
the algorithm will prefer the copy of cost $c_{e}$ if it takes only one of the
two elements into the solution. This again implies that we obtain a one to one
correspondence of solutions as in the discussion above.

\begin{theorem}
    $\copic(\base(\M_{1}), \base(\M_{2}), \diag(a), c, c)$ can be solved in
    polynomial time, for any two matroids $\M_{1}, \M_{2}$ and cost vectors $a
    \in \r^{n}_{\geq 0}, c \in \r^{n}$.
\end{theorem}

It remains an interesting open question whether we can also solve the case
with arbitrary costs $a \in \r^{n}$ and the case with non-equal linear costs $c
\neq d$ in polynomial time, as it is possible for uniform and partition
matroids.

\subsection{Pairs of paths}
\label{sec:diagonal-paths}

In this section we analyze the special case when $\f_{1}$ and $\f_{2}$
correspond to the set of $s_{1}$-$t_{1}$- and $s_{2}$-$t_{2}$-paths in a graph. We
will again look at the case where the graphs corresponding to $\f_{1}$ and
$\f_{2}$ are identical. One must also make sure that there do not exist negative
circles in the graph, else already optimizing over a linear cost function
without interaction costs is NP-hard. To
simplify the exposition we will focus on $Q,c,d \geq 0$. Table \ref{tab:paths} is a summary of the results in this subsection.
It is important to differentiate between directed and undirected graphs,
which is clear in the light of Proposition~\ref{paths-eq-copic} and the known complexity results of the edge-disjoint
paths problem.

\begin{prop}\label{paths-eq-copic}
    Given a graph $G$, $\copic(\paths_{s_{1}, t_{1}}(G), \paths_{s_{2}, t_{2}}(G), \diag(a), 0, 0)$
    with $a > 0$ has a solution 
    with objective value $0$, if and only if there exist two edge-disjoint paths
    $s_{i}$-$t_{i}$-paths in $G$.
\end{prop}

\begin{table}[h]
\begin{tabular}{l  l  l  l}
    \textbf{directedness} & \textbf{terminals} & \textbf{cost restrictions} &
    \textbf{complexity} \\ \hline
    directed & arbitrary & $Q=I, c=d=0$ & NP-hard \\
    directed & common & $Q=\diag(\infty)$ & NP-hard \\
    undirected & arbitrary & $Q=\diag(\infty), d=0$ & NP-hard \\
    undirected & arbitrary & $c=d=0$ & open \\
    undirected & common & $Q=\diag(\infty)$ & NP-hard \\
    both & common & $c=d$ & P\\
\end{tabular}
	\bigskip
\caption{Summary of the results for diagonal COPIC with paths as feasible solutions}
\label{tab:paths}
\end{table}

It is well known that the edge-disjoint paths problem is polynomial time solvable
for every constant number of paths in undirected
graphs~\cite{robertson1995graph}, but NP-hard already
for 2 paths in directed graphs~\cite{fortune1980directed}. This imediatly yields
the following result.

\begin{corollary}
    Given a directed graph $G$, $\copic(\paths_{s_{1}, t_{1}}(G), \paths_{s_{2},  t_{2}}(G), \diag(a), 0, 0)$ is NP-hard, even for $a \equiv 1$.
\end{corollary}

We use the following results obtained by
Eilam-Tzoreff~\cite{eilam1998disjoint} to further classify the complexity of our
problem.
\begin{theorem}[Eilam-Tzoreff~\cite{eilam1998disjoint}]\label{eilam-poly}
    The undirected edge-disjoint two shortest paths problem is polynomial time
    solvable, even in the weighted case.
    On the other hand, the undirected two edge-disjoint one shortest paths problem is NP-hard.
\end{theorem}
It is important to note that in the results of Eilam-Tzoreff, a shortest path
always means a shortest path in the original graph, not a shortest path after removing 
the edges of the other disjoint path.
This is the reason why using Theorem~\ref{eilam-poly},
we cannot conclude that $\copic(\paths_{s_{1}, t_{1}}(G),
\paths_{s_{2}, t_{2}}(G), \diag(\infty), c, c)$  is polynomial time
solvable, since in our model we cannot enforce two shortest paths of the original graph.
If $c=d=1$ and $Q = \diag(\infty)$ Bj\"orklund and
Husfeldt~\cite{bjorklund2014shortest} showed in 2014 how to solve the problem using a
polynomial time Monte Carlo algorithm. The existence of a deterministic
polynomial time algorithm is still unknown and a long-standing open problem.

Nevertheless, it is possible to use the hardness results of
Eilam-Tzoreff~\cite{eilam1998disjoint} to show that for general costs $c, d \geq
0$ the problem is NP-hard.

\begin{corollary}
    Given an undirected graph $G$, $\copic(\paths_{s_{1}, t_{1}}(G), \paths_{s_{2}, t_{2}}(G), \diag(\infty),
    c, 0)$ is NP-hard for $c \geq 0$.
\end{corollary}
\begin{proof}
    Using a polynomial time algorithm for $\copic$ we can determine, if the two
    edge-disjoint one shortest paths problem has a solution. Just run the
    algorithm and check if the objective value equals the length of a shortest
    $s_{1}$-$t_{1}$-path in the given graph.
\end{proof}

This covers the case if $s_{1} \neq s_{2}$ and $t_{1} \neq t_{2}$. From the
edge-disjoint path literature we know that the problem becomes easier, if one
assumes a common source $s$ and a common sink $t$ for all the paths. We can classify the
complexity of this case for our problem, using the following results.

\begin{theorem}
    Given a graph or digraph $G$, $\copic(\paths_{s, t}(G), \paths_{s, t}(G),
    \diag(a), c, c)$ is solvable in polynomial time, for cost vectors $a,c \geq 0$.
\end{theorem}
\begin{proof}
    We reduce to a minimum cost flow problem. Set $b(s) = 2$ and $b(t) = -2$ and
    double each edge/arc $e \in E$ to two versions $e_{1}, e_{2}$ with
    $\tilde{c}_{e_{1}} = c_e$ and $\tilde{c}_{e_{2}} = a_e + c_e$. Now a
    minimum cost flow in this network will be integral and can be decomposed into two path flows, each sending one unit from $s$ to $t$. The cost of the flow
    corresponds to the cost of these two paths in our problem.
\end{proof}

\begin{theorem}
    Given a graph or digraph $G$, $\copic(\paths_{s, t}(G), \paths_{s, t}(G),
    \diag(\infty), c, d)$ is NP-hard.
\end{theorem}
\begin{proof}
    For digraphs the statement follows from a reduction from directed two disjoint paths. Given such
    an instance we introduce the new terminals $s$ and $t$ and add arcs $(s,
    s_{1}),$ $(s, s_{2}),$ $(t_{1}, t),$ $(t_{2}, t)$. We use $Q = \diag(\infty)$ and
    as linear costs $c_{(s, s_{1})} = c_{(t_{1}, t)} = d_{(s, s_{2})} = d_{(t_{2}, t)} = 0$ and $c_{(s, s_{2})} = c_{(t_{2}, t)} = d_{(s, s_{1})} = d_{ (t_{1}, t)} = \infty$ and $c_e = d_e = 0$
    for all other edges. This enforces that paths $S_{i}$ are
    $s_{i}$-$t_{i}$-paths and the diagonal matrix with infinite entries ensures
    disjointness.

    In the undirected case we apply the same construction as above but using the undirected two edge-disjoint one shortest paths problem. 
	To solve the decision problem analyzed by Eilam-Tzoreff~\cite{eilam1998disjoint}, we create COPIC with
    $c_e = 1$ and $d_e = 0$ for all the edges in the original network to
    enforce that $S_{1}$ is a shortest path. After finding a finite cost
    solution to this problem we check if the length of $S_{1}$ is equal to the
    length of a shortest $s_{1}$-$t_{1}$-path in $G$.
\end{proof}

\section{Linearizable instances}\label{sec:lin}

In this section we explore for which cost matrices COPIC leads to an equivalent problem where there is essentially no interaction between two structures of COPIC.  

More precisely, we say that an interaction cost matrix $Q$ of a COPIC is \emph{linearizable}, if there exist vectors $a=(a_i)$ and $b=(b_i)$ such that for all $S_1\in\f_1$ and $S_2\in\f_2$ 
\[
	\sum_{i\in S_1}\sum_{j\in S_2}q_{ij}=\sum_{i\in S_1}a_i+\sum_{j\in S_2}b_j
\]
holds. In that case we say that the pair of vectors $a$ and $b$ together is a \emph{linearization} of $Q$.

Note that for an instance $\copic(\f_1,\f_2,Q,c,d)$, $f(S_1,S_2)=\sum_{i\in S_1}\bar{a}_i+\sum_{j\in S_2}\bar{b}_j$ for some $\bar{a}=(\bar{a}_i)$, $\bar{b}=(\bar{b}_i)$ and all $S_1\in \f_1$, $S_2\in \f_2$, if and only if $Q$ is linearizable. Hence, we extend our notion of linearizability and say that an instance $\copic(\f_1,\f_2,Q,c,d)$ is linearizable if and only if $Q$ is linearizable. 
Our aim is to characterize all linearizable instances of  COPIC, with respect to given solution sets $\f_1$ and $\f_2$.

Linearizable instances have been studied by various authors for the case of quadratic assignment problem \cite{CDW16, KP11, PK13}, quadratic spanning tree problem \cite{CP15} and bilinear assignment problem \cite{CSPB16}. Here we generalize the ideas from \cite{CSPB16} and suggest an approach for finding a characterization of linearizable instances of COPIC's.

An interaction cost matrix $Q$ of a COPIC has \emph{constant objective property with respect to $\f_1$} if for every $j\in [n]$ there exist a constant $K^{(1)}_j$, so that
\[
	\sum_{i\in S_1}q_{ij}=K^{(1)}_j \qquad \text{ for all }\ S_1\in\f_1.
\]
Similarly, $Q$ has \emph{constant objective property with respect to $\f_2$}  if for every $i\in [m]$ there exist a constant $K^{(2)}_i$, so that
\[
	\sum_{j\in S_2}q_{ij}=K^{(2)}_i \qquad \text{ for all }\ S_2\in\f_2.
\]
For $\f_i$, $i=1,2$, let  $\text{CVP}_i(\f_i$) be
the vector space of all matrices with constant objective property with respect to $\f_i$.


Combinatorial optimization problems with constant objective property have been studied by various authors \cite{Berenguer79, Burkard07, CusticKlinz16, Kaveh10}.

Let $\text{CVP}_1(\f_1)+\text{CVP}_2(\f_2)$ be the vector space of all interaction matrices $Q=(q_{ij})$ of COPIC, such that $q_{ij}=a_{ij}+b_{ij}$ $\forall i,j$, for some $A=(a_{ij})\in \text{CVP}_1(\f_1)$ and  $B=(b_{ij})\in \text{CVP}_2(\f_2)$.

\begin{lemma}[Sufficient conditions]\label{lm:sufLin}
If the interaction cost matrix $Q$ of $\copic(\f_1,\f_2,Q,c,d)$ is an element of $\text{CVP}_1(\f_1)+\text{CVP}_2(\f_2)$, then $Q$ is linearizable.
\end{lemma}
\begin{proof}
	Let $Q$ be of the form $Q=E+F$, where $E=(e_{ij})\in \text{CVP}_1(\f_1)$ and $F=(f_{ij})\in \text{CVP}_2(\f_2)$. Then 
\begin{align*}
	\sum_{i\in S_1}\sum_{j\in S_2}q_{ij} & = \sum_{i\in S_1}\sum_{j\in S_2} \left(e_{ij}+f_{ij}\right)\\
		& = \sum_{j\in S_2}\left(\sum_{i\in S_1} e_{ij}\right)+ \sum_{i\in S_1}\left(\sum_{j\in S_2} f_{ij}\right)\\
		& = \sum_{j\in S_2} K^{(1)}_j+ \sum_{i\in S_1}K^{(2)}_i.
\end{align*}
Hence $Q$ is linearizable, and $a=(a_i)$, $b=(b_j)$ with $a_i=K^{(2)}_i$, $b_j=K^{(1)}_j$ is a linearization of $Q$.
\end{proof}

Now we show that the opposite direction is also true, provided some additional conditions are satisfied. In fact, these additional conditions are satisfied for many well studied combinatorial optimization problems.

\begin{lemma}[Necessary conditions]\label{lm:nesLin}
Let $\f_1\subseteq 2^{[m]}$ and $\f_2\subseteq 2^{[n]}$ be such that:
\begin{enumerate}[label=(\roman*)]
\item There exist an $m$ vector $a=(a_i)$, an $n$ vector $b=(b_j)$ and two non-zero constants $K_a,K_b$, such that 
	\[\sum_{i\in S_1}a_i=K_a  \ \ \forall S_1\in \f_1\  \ \text{ and } \ \ \sum_{j\in S_2}b_j=K_b \ \ \forall S_2\in \f_2.
\]
\item If an $m\times n$ matrix $\bar{Q}=(\bar{q}_{ij})$ is such that            $\sum_{i\in S_1}\sum_{j\in S_2} \bar{q}_{ij}=0$ for all $S_1\in\f_1$, $S_2\in\f_2$, then $\bar{Q}\in \text{CVP}_1(\f_1)+\text{CVP}_2(\f_2)$.
\end{enumerate}
If $\copic(\f_1,\f_2,Q,c,d)$ is linearizable, then $Q\in \text{CVP}_1(\f_1)+\text{CVP}_2(\f_2)$.
\end{lemma}
\begin{proof}
	Assume that the conditions $(i)$ and $(ii)$ of Lemma~\ref{lm:nesLin} are satisfied, and that $Q$ is linearizable. We will show that $Q\in \text{CVP}_1(\f_1)+\text{CVP}_2(\f_2)$ by reconstructing the proof of Lemma~\ref{lm:sufLin} in reverse direction.

Since $Q$ is linearizable, there exist $a=(a_i)$ and $b=(b_j)$ such that 
\begin{equation}\label{eq:nes11}
\sum_{i\in S_1}\sum_{j\in S_2}q_{ij}=\sum_{i\in S_1} a_i + \sum_{j\in S_2}b_j \quad \forall S_1\in \f_i, S_2\in \f_2.
\end{equation}
Note that from $(i)$ it follows that there exist matrices $\hat{E}=(\hat{e}_{ij})\in \text{CVP}_1(\f_1)$ and $\hat{F}=(\hat{f}_{ij})\in \text{CVP}_2(\f_2)$ such that
\begin{equation}\label{eq:nes22}
	\sum_{j\in S_2}\hat{f}_{ij}=a_i \qquad \forall S_2\in\f_2,\ i\in M,
\end{equation}
\begin{equation}\label{eq:nes33}
	\sum_{i\in S_i}\hat{e}_{ij}=b_j \qquad \forall S_1\in\f_1,\ j\in N.
\end{equation}
Using \eqref{eq:nes22} and \eqref{eq:nes33}, we can rewrite \eqref{eq:nes11} as
\begin{align}
	\sum_{i\in S_1}\sum_{j\in S_2}q_{ij} & = \sum_{i\in S_1}\left( \sum_{j\in S_2}\hat{f}_{ij} \right) + \sum_{j\in S_2}\left( \sum_{i\in S_1}\hat{e}_{ij} \right)\nonumber\\
	& = \sum_{i\in S_1}\sum_{j\in S_2} \left(\hat{e}_{ij}+\hat{f}_{ij}\right)
\end{align}
for all $S_1\in\f_1$, $S_2\in\f_2$. Hence it follows that 
\begin{equation}
	\sum_{i\in S_1}\sum_{j\in S_2}\left(q_{ij}-(\hat{e}_{ij}+\hat{f}_{ij})\right)=0 \quad \forall S_1\in \f_i, S_2\in \f_2.
\end{equation}
Now, from $(ii)$ it follows that $Q-(\hat{E}+\hat{F})=E+F$ for some $E\in \text{CVP}_1(\f_1)$, $F\in \text{CVP}_2(\f_2)$, and hence, $Q=(E+\hat{E})+(F+\hat{F})\in \text{CVP}_1(\f_1) + \text{CVP}_2(\f_2)$. 
\end{proof}

From Lemma~\ref{lm:sufLin} and Lemma~\ref{lm:nesLin} it follows that $\text{CVP}_1(\f_1) + \text{CVP}_2(\f_2)$ is the set of all linearizable matrices, provided that the corresponding COPIC satisfies properties $(i)$ and $(ii)$ of Lemma~\ref{lm:nesLin}.

In most cases, property $(i)$ is straightforward  to check. For example, it is true for all COPIC's for which elements of $\f_1$ and $\f_2$ are of fixed cardinality. If $\f_1$ and $\f_2$ are $s$-$t$ paths in a graph, then again property $(i)$ is satisfied, although feasible solutions are of different cardinality. Condition $(i)$ is not satisfied for unconstrained solution sets, i.e., when $\f_1$ ($\f_2$) is $2^{[m]}$ ($2^{[n]}$).
\medskip

Now we show how Lemma~\ref{lm:sufLin} and Lemma~\ref{lm:nesLin} can be used to characterize linearizable instances for some specific COPIC's. In particular, we consider unconstrained solution sets $2^{[m]}$, bases of the uniform matroids $\base(\umat{m}{k})$, spanning trees of a complete graph $\base(\mathcal{M}(K_m))$ and perfect matchings of a complete bipartite graph $\pmatch(K_{m,m})$. For the case of  $\pmatch(K_{m,m})$ the set $[m]\times [m]$ will be our set of edges of the perfect bipartite graph $K_{m,m}$. 
Hence, in the case of $\copic(\f_1,\f_2,Q,c,d)$ where $\f_i=\pmatch(K_{m,m})$, the dimensions (number of indices) of the cost arrays $Q$ and $c$ or $d$ is increased by one, however our lemmas and $\text{CVP}_i(\pmatch(K_{m,m}))$ remain to be well defined.

\begin{theorem}\
\begin{enumerate}[label=(\roman*)]
	\item $\copic(\pmatch(K_{m,m}),\pmatch(K_{n,n}),Q,c,d)$ is linearizable if and only if there are some arrays $A$, $B$, $C$, $D$ such that $q_{ijk\ell}=a_{ijk}+b_{ij\ell}+c_{ik\ell}+d_{jk\ell}$.
	\item $\copic(\base(\mathcal{M}(K_m)),\base(\mathcal{M}(K_n)),Q,c,d)$ is linearizable if and only if there are some vectors $a$, $b$ such that $q_{ij}=a_{i}+b_{j}$.
	\item $\copic(\base(\umat{m}{k_1}),\base(\umat{n}{k_2}),Q,c,d)$ is linearizable if and only if there are some vectors $a$, $b$ such that $q_{ij}=a_{i}+b_{j}$.
	\item $\copic(\pmatch(K_{m,m}),\base(\mathcal{M}(K_n)),Q,c,d)$ is linearizable if and only if there are some arrays $A$, $B$, $C$ such that $q_{ijk}=a_{ij}+b_{ik}+c_{jk}$.
	\item $\copic(\base(\mathcal{M}(K_m)) ,\base(\umat{n}{k}) ,Q,c,d)$ is linearizable if and only if there are some vectors $a$, $b$ such that $q_{ij}=a_{i}+b_{j}$.
	\item $\copic(\pmatch(K_{m,m}) ,\base(\umat{n}{s}) ,Q,c,d)$ is linearizable if and only if there are some arrays $A$, $B$, $C$ such that $q_{ijk}=a_{ij}+b_{ik}+c_{jk}$.
\end{enumerate}
\end{theorem}
\begin{proof}
We present a  complete proof for $(iv)$, and indicate how other statements can be shown analogously. 

In the case of $\copic(\pmatch(K_{m,m}),\base(\mathcal{M}(K_n)),Q,c,d)$, the
interaction costs are represented in a three-dimensional array $Q$, since for
convenience we represent the cost vector of $\f_1=\pmatch(K_{m,m})$ in two indices.
It is well known that a linear assignment problem instance $R=(r_{ij})$ has the
constant objective property if and only if $r_{ij}=s_i+t_j$, for some vectors $s$
and $t$. Hence $\text{CVP}_1(\pmatch(K_{m,m}))=\{A=(a_{ijk})\colon
a_{ijk}=b_{ik}+c_{jk} \text{ for some } B=(b_{ij}),C=(c_{ij})\}$. A spanning tree
problem on a complete graph has the constant objective property if and only if the cost vector is constant, therefore $\text{CVP}_2(\base(\mathcal{M}(K_n)))=\{A=(a_{ijk})\colon a_{ijk}=b_{ij} \text{ for some } B=(b_{ij})\}$. Hence, $Q$ is an element of $\text{CVP}_1(\pmatch(K_{m,m}))+\text{CVP}_2(\base(\mathcal{M}(K_n)))$ if and only if  there are some $A$, $B$ and $C$ such that
\begin{equation}\label{STlin}
	q_{ijk}=a_{ij}+b_{ik}+c_{jk}.
\end{equation} 

Lemma~\ref{lm:sufLin} tells us that \eqref{STlin} is a sufficient condition for $Q$ to be linearizable. To show that it is also a necessary condition, we just need to show that properties $(i)$ and $(ii)$ of Lemma~\ref{lm:nesLin} are true for $\copic(\pmatch(K_{m,m}),\base(\mathcal{M}(K_n)),Q,c,d)$. $(i)$ is obviously true, hence it remains to show that if $Q$ is such that
\[
	\sum_{(i,j)\in S_1}\sum_{k\in S_2} q_{ijk}=0\qquad \forall S_1\in\pmatch(K_{m,m}),\ S_2\in\base(\mathcal{M}(K_n)),
\]
then $Q\in \text{CVP}_1(\pmatch(K_{m,m}))+\text{CVP}_2(\base(\mathcal{M}(K_n)))$.

Let $i,j\in \{2,3,\ldots,m\}$ be fixed, and let $S_{PM}',S_{PM}''\in\pmatch(K_{m,m})$ be such that $S_{PM}'\setminus S_{PM}''=\{(1,1),(i,j)\}$ and $S_{PM}''\setminus S_{PM}'=\{(1,j),(i,1)\}$. 
Further, let $k\in \{2,3,\ldots,n\}$ be fixed, and $S_{ST}',S_{ST}''\in\base(\mathcal{M}(K_n))$ be such that $S_{ST}'\setminus S_{ST}''=\{1\}$ and $S_{ST}''\setminus S_{ST}'=\{k\}$. Note that such $S_{PM}',S_{PM}'',S_{ST}',S_{ST}''$ exist for all $i,j\in \{2,3,\ldots,m\}$, $k\in \{2,3,\ldots,n\}$. 

Let us assume that $Q$ satisfies property $(ii)$ of Lemma~\ref{lm:nesLin}. Then, in particular, we have that
\begin{equation}\label{eqcor0}
	\sum_{(i,j)\in S_{PM}'}\sum_{k\in S_{ST}'} q_{ijk}+\sum_{(i,j)\in S_{PM}''}\sum_{k\in S_{ST}''} q_{ijk}=\sum_{(i,j)\in S_{PM}'}\sum_{k\in S_{ST}''} q_{ijk}+\sum_{(i,j)\in S_{PM}''}\sum_{k\in S_{ST}'} q_{ijk},
\end{equation}
which, after cancellations, gives us
\begin{equation}\label{eqcor1}
	q_{111}+q_{ij1}+q_{1jk}+q_{i1k}=q_{11k}+q_{ijk}+q_{1j1}+q_{i11} 
\end{equation}
for all $i,j\in \{2,3,\ldots,m\}$, $k\in \{2,3,\ldots,n\}$. Note that \eqref{eqcor1} holds true even if $i,j$ or $k$ is equal to $1$, since in that case everything chancels out. Therefore, $q_{ijk}$ can be expressed as
\begin{equation}
q_{ijk}=a_{ij}+b_{ik}+c_{jk} \qquad \forall i,j\in [m],\ \forall k\in [n],
\end{equation}
where 
\[a_{ij}:=q_{ij1}-\frac{1}{2}q_{1j1}-\frac{1}{2}q_{i11}+\frac{1}{3}q_{111},\]
\[b_{ik}:=q_{i1k}-\frac{1}{2}q_{11k}-\frac{1}{2}q_{i11}+\frac{1}{3}q_{111},\]
\[c_{jk}:=q_{1jk}-\frac{1}{2}q_{11k}-\frac{1}{2}q_{1j1}+\frac{1}{3}q_{111},\]
i.e., $Q\in \text{CVP}_1(\pmatch(K_{m,m}))+\text{CVP}_2(\base(\mathcal{M}(K_n)))$. That proves statement $(iv)$ of the theorem.

Statements $(i)$ and $(ii)$ of the theorem can be proved by considering equation \eqref{eqcor0} with two pairs of 
$S_{PM}',S_{PM}''$ for the case of $\copic(\pmatch(K_{m,m}),\pmatch(K_{n,n}),Q,c,d)$, and two pairs of $S_{ST}',S_{ST}''$ for the case of $\copic(\base(\mathcal{M}(K_m)),\base(\mathcal{M}(K_n)),Q,c,d)$. Using analogous approach, the remaining statements of the theorem can be shown.
\end{proof}

As we mentioned before, property $(i)$ of Lemma~\ref{lm:nesLin} does not hold for unconstrained solution set $2^{[m]}$ ($2^{[n]}$), nevertheless, it is not hard to show that $\text{CVP}_1(\f_1) + \text{CVP}_2(\f_2)$ characterizes all linearizable matrices even if $\f_1=2^{[m]}$ or $\f_2=2^{[n]}$.

\begin{theorem}
	$\copic(\f_1,\f_2,Q,c,d)$ with $\f_1=2^{[m]}$ ($\f_2=2^{[n]}$) is linearizable if and only if  $Q\in\text{CVP}_2(\f_2)$ ($Q\in\text{CVP}_1(\f_1)$).
\end{theorem}
\begin{proof}
	Assume that $\f_1=2^{[m]}$.
	Note that $\text{CVP}_1(2^{[m]})$ contains only the $m\times n$ zero matrix, hence Lemma~\ref{lm:sufLin} implies that elements of $\text{CVP}_2(\f_2)$ are linearizable. 

Now let us assume that $Q$ is linearizable and not an element of $\text{CVP}_2(\f_2)$. Then there must exist some $i'\in [m]$ and $S_2,S_2'\in\f_2$ such that 
$\sum_{j\in S_2}q_{i'j}\neq \sum_{j\in S_2'}q_{i'j}$. 
Let $a=(a_i)$ and $b=(b_i)$ be a linearization of $Q$. Since $\{i'\}\in 2^{[m]}$, we have that
\[
	\sum_{j\in S_2}q_{i'j}=\sum_{i\in \{i'\}}\sum_{j\in S_2}q_{ij}=a_{i'}+\sum_{j\in S_2}b_j,
\]
\[
	\sum_{j\in S_2'}q_{i'j}=\sum_{i\in \{i'\}}\sum_{j\in S_2'}q_{ij}=a_{i'}+\sum_{j\in S_2'}b_j.
\]
Hence, $\sum_{j\in S_2}b_j\neq \sum_{j\in S_2'}b_j$. However, since $\emptyset \in 2^{[m]}$ we have
\[
	0=\sum_{i\in \emptyset}\sum_{j\in S_2}q_{ij}=\sum_{j\in S_2}b_j \quad \text{ and }\quad
	0=\sum_{i\in \emptyset}\sum_{j\in S_2'}q_{ij}=\sum_{j\in S_2'}b_j
\]
which implies that $\sum_{j\in S_2}b_j= \sum_{j\in S_2'}b_j$, a contradiction.
\end{proof}

\section{Conclusion}\label{sec:conclusion}

We introduced a general model to study combinatorial optimization problems with
interaction costs and showed that many classical hard combinatorial optimization problems are special cases. In many cases, interaction costs can be identified as the origin of the
hardness of these problems. Therefore we
considered special structures of interaction costs, and their
impact on the computational complexity of the underlying combinatorial optimization problems. 
We presented a general approach based on multi-parametric programming to solve
instances parametrized with the rank of the interaction cost matrix $Q$.
Complementary to that, we analyzed problems with diagonal interaction cost matrix
$Q$, which can be used to enforce disjointness constraints. Even for
this special type of interaction costs, we can show that for many common sets of
feasible solutions, that have no matroid structure, COPIC becomes NP-hard.
We also identified conditions on the interaction costs so that COPIC
can be reduced to an equivalent instance with no interaction costs.

To further characterize how interaction costs impact the computational
complexity of different combinatorial optimization problems, the following
questions could be addressed.

\begin{enumerate}
    \item Are the polynomially solvable cases of COPIC where matrix $Q$ has
        fixed rank $r$ W[1]-hard?
    \item For cases of COPIC with diagonal matrix that can be efficiently
        solved, analyze the parameterized complexity with respect to the
        bandwith of $Q$.
    \item Can $\copic(\base(\umat{m}{k}), \paths_{s,t}(G), \diag(a), c, d)$ be
        solved in polynomial time, if $a \geq 0, c \geq 0$ and $d \geq 0$?
    \item Is $\copic(\base(\M_{1}), \base(\M_{2}), \diag(a), c, d)$ solvable in
        polynomial time, without any restrictions on $\M_{1}, \M_{2}, a, c$ and
        $d$?
\end{enumerate}

For the case of diagonal COPIC it would be interesting to
study further types of sets of feasible solutions. For example the matching-cut problem
analyzed by Bonsma~\cite{bonsma2009complexity} can be also formulated as a
special case of diagonal COPIC, so analyzing graph cuts as feasible sets in
diagonal COPIC is an interesting candidate for further research.

Additionally, understanding the influence of interaction costs with other special matrix structures,
besides fixed rank and diagonal matrices, to the computational complexity of
combinatorial optimization problems would be of interest.

\bibliographystyle{abbrv}
\bibliography{bipartite-comb}

\end{document}